\documentclass[twoside,11pt]{article}

\usepackage{jmlr2e}

\usepackage{times}

\usepackage{soul}
\usepackage{url}
\usepackage{hyperref}
\usepackage[utf8]{inputenc}
\usepackage[small]{caption}
\usepackage{graphicx}
\graphicspath{{sim/Energy_harvesting/figure/}{sim/Single_ml/figure/}}
\usepackage{amsmath}
\usepackage{booktabs}
\usepackage{subfigure}
\usepackage{amssymb}
\usepackage{bm}
\usepackage{algorithm}  
\usepackage{algorithmicx}  
\usepackage{algpseudocode}  
\usepackage{xcolor}

\renewcommand{\cite}[1]{\citep{#1}}

\ShortHeadings{Provably Efficient Model-Free Algorithm  for MDPs with Peak Constraints}{Bai, Aggarwal, and Gattami}
\firstpageno{1}
\begin{document}
	
\title{Provably Sample-Efficient Model-Free Algorithm for MDPs with Peak Constraints}
	
\author{\name Qinbo Bai \email bai113@purdue.edu \\
		\addr School of Electrical and Computer Engineering\\
		Purdue University\\
		West Lafayette, IN 47907, USA
		\AND
		\name Vaneet Aggarwal\email vaneet@purdue.edu \\
		\addr School of IE and ECE\\
		Purdue University\\
		West Lafayette, IN 47907, USA
		\AND
		\name Ather Gattami \email ather.gattami@ai.se \\
		\addr AI Sweden\\
		Stockholm, Sweden}
	
\editor{}
	
\maketitle

\begin{abstract}
	In the optimization of dynamic systems, the variables typically have constraints. Such problems can be modeled as a Constrained Markov Decision Process (CMDP). This paper considers the peak Constrained Markov Decision Process (PCMDP), where the agent chooses the policy to maximize total reward in the finite horizon as well as satisfy constraints at each epoch with probability 1. We propose a model-free algorithm that converts PCMDP problem to an unconstrained problem and a Q-learning based approach is applied. We define the concept of probably approximately correct (PAC) to the proposed PCMDP problem. The proposed algorithm is proved to achieve an  $(\epsilon,p)$-PAC policy when the episode $K\geq\Omega(\frac{I^2H^6SA\ell}{\epsilon^2})$, where $S$ and $A$ are the number of states and actions, respectively. $H$ is the number of epochs per episode. $I$ is the number of constraint functions, and $\ell=\log(\frac{SAT}{p})$. We note that this is the first result on PAC kind of analysis for  PCMDP with peak constraints, where the transition dynamics are not known apriori. We demonstrate the proposed algorithm on an energy harvesting problem and a single machine scheduling problem, where it performs close to the theoretical upper bound of the studied optimization problem.
\end{abstract}
\section{Introduction}

Optimization of dynamic systems typically has constraints, e.g., battery capacity for robots. As an example, if a robot is powered by a battery, which is also being charged with an external power supply, the amount of energy used at each time is limited by the battery capacity. The dynamical systems are typically modeled as a Markov Decision Process (MDP), while the transition probabilities may not be known apriori (or maybe dynamic). In the absence of knowledge of transition probabilities, the MDP is modeled as a Reinforcement Learning (RL) problem which aims to maximize the total reward in the finite horizon by making actions given the state of the process to be controlled. RL algorithms can be divided into model-based and model-free, where the model-based approaches estimate the transition probabilities, while model-free approaches do not. In this paper, we consider a model-free approach to RL in the presence of peak constraints, which is an important constraint in many dynamical systems. For instance, algorithms with peak constraints have been studied for communications \cite{shamai1995capacity}, flow-shop scheduling \cite{fang2013flow}, thermostatically-controlled systems \cite{karmakar2013coordinated}, economics \cite{bailey1972peak}, robotics \cite{li1997new}, etc. The constrained optimization problems have been considered for Markov Decision Processes \cite{altman1999constrained}. However, these require complete knowledge of the transition probabilities. Without such knowledge, algorithms have been proposed \cite{geibel2005risk,geibel2006reinforcement}. However, to the best of our knowledge, none of the algorithms so far has considered MDPs with peak constraints and provably given PAC analysis for objective and constraint violations. 

{\bf Contributions: }  In this work,  we assume that the transition probability is unknown, the reward and constraint functions can be observed but are not known in closed form. We extend the concept of probably approximately correct (PAC) for the proposed PCMDP problem. We introduce an \textbf{Approximated PCMDP} to propose a novel model-free algorithm with stochastic policy. The proposed algorithm is shown to achieve an $(\epsilon,p)$-PAC policy when the number of episodes are  $K\geq\Omega(\frac{I^2H^6SA\ell}{\epsilon^2})$. We also show that a deterministic variant of the proposed algorithm, while converges to the optimal, may not have optimal convergence rates. We conjecture that a deterministic policy does not achieve the optimal convergence rate, while validating that is left for the future. Finally, the proposed algorithm is evaluated on an energy-harvesting transmitter studied in \cite{wang2014power} and a single machine scheduling problem with deadlines studied in \cite{KOULAMAS2001447}. It is found that the proposed algorithm performs close to the genie-aided upper bound for the problem.

\section{Related Work}

\textbf{Online Convex Optimization (OCO)}: OCO problem is an extension of the constrained convex optimization. In this problem, we wish to optimize $\sum_{t=1}^{T}f_t(\bm{x})$ for given functions $f_t$, $t\in \{1, \cdots, T\}$, such that $\bm{x}\in\mathcal{K}$. In online convex optimization, we select $\bm{x}_t$ at time $t$, such that the regret in objective is minimized, which is defined as 

\begin{equation}
	\textup{Regret}(T)=\sum_{t=1}^{T}f_t(\bm{x}_t)-\min\limits_{\bm{x}\in\mathcal{K}}\sum_{t=1}^{T}f_t(\bm{x}). 
\end{equation}

Further, $\bm{x}_t$ may not satisfy constraints, and thus there will be a constraint violation. 
By changing the problem into an online convex-concave optimization problem, 
The authors of \cite{OCOLTC} proposed an algorithm which achieves the $O(\sqrt{T})$ bound for the regret and $O(T^{3/4})$ bound on the violation of constraints. 
Further, they proposed another algorithm based on the mirror-prox method \citep{nemirovski2004prox} that achieves $O(T^{2/3})$ bound on both regret and constraints when the domain can be described by a finite number of linear constraints. 
The authors of \cite{jenatton2016adaptive} proposed an algorithm which achieves $O(T^{\max(\beta,1-\beta)})$ objective regret and $O(T^{1-\beta/2})$ constraint violations for $\beta\in (0,1)$. 
Further, the authors of \citep{OCOFCV} proposed an algorithm with $O(\sqrt{T})$  regret bound for objective with finite constraint violations. However, CMDP is different from OCO because the reward function depends both on the state and action and the previous action can influence current state and thus change the reward function. Further, the functions and constraints are not known explicitly in reinforcement learning (RL). Thus, the analysis of CMDP doesn't directly follow from that of OCO.

\textbf{Constrained Markov Decision Process (CMDP)}: When the system model (the transition probability distribution, the reward function, and constraint functions) is known, the problem is generally considered as CMDP. CMDP in the form of discounted and average reward has been  studied in \cite{altman1999constrained}. It is well known that CMDP problem is convex and can be converted into an equivalent unconstrained MDP problem by using the method of Lagrange multipliers. Thus, when the model is known, CMDP can be solved using linear programming (LP). In addition to the LP method, Three different algorithms, WeiMDP, AugMDP, and RecMDP, are proposed to solve CMDP in different settings \cite{geibel2006reinforcement}. Constrained Upper Confidence Reinforcement Learning (C-UCRL) \cite{zheng2020constrained} algorithm achieves $O(T^{3/4}\sqrt{\log(T/\delta)})$ regret on the reward and satisfies the constrains even during learning process with probability at least $1-\delta$. 

The key difference between these works and reinforcement learning (RL) approach is that the transition probabilities for the next state given the previous state and action are assumed to be known in CMDP approaches, while are not known apriori in RL approaches. They may be learnt in model-based RL, while not learnt at all in model-free RL approaches. Recently, the authors of \cite{efroni2020explorationexploitation} proposed OptCMDP, OptCMDP-bonus, OptDual-CMDP, and OptPrimalDual-CMDP algorithms to achieve both $O(\sqrt{T})$ bound on the reward and constraint violations. However, all of these four algorithms are model-based. A model-free algorithm is provided in \cite{ding2020provably} and also achieve $\mathcal{O}(\sqrt{T})$ bound on both the objective and constraint violation. In this paper, we consider model-free reinforcement learning based approaches for the peak constraint problem which is different from the above paper in the average expected setting. 

\textbf{Regret Bounds for Reinforcement Learning}: Regret Analysis for the Reinforcement Learning has been considered for both the model-based approaches \cite{jaksch2010near,agrawal2017optimistic,Azar,kakade2018variance} and the model-free approaches \cite{kearns2002near,strehl2006pac,jinchi}. Our paper extends the epsiodic reinforcement learning setup with the addition of peak constraints. 

\textbf{Model-free Reinforcement Learning Algorithm for CMDP with Peak Constraints}: Q-learning based methods with peak constraints has been studied \cite{bouton,wang2014power}, where the Q function in each epoch is projected to the constraint set. These algorithms involve knowledge of constraint functions explicitly (since projection to the constraint set is needed) to make decisions at each time. In contrast, we do not require knowledge of constraint function.  Recently, based on the primal-dual method, The authors of \cite{PRIMALDUAL} proposed an algorithm with policy descent and showed that the algorithm is $1-\delta$ safe ( $P(\cap_{t\geq0}\{s_t\in\mathcal{S}_0\}\vert\pi_{\theta})\geq1-\delta$, where $\mathcal{S}_0$ is the safe region). Besides, \cite{ather} related PCMDP to unconstrained zero-sum game where the objective is the Lagrangian of the optimization problem, and applied max-min Q-learning to PCMDP to prove convergence. However, none of the works in this direction have given a PAC kind of analysis for objectives and constraints, which is the focus of our paper. To the best of our knowledge, this paper provides  the first PAC analysis for model-free reinforcement learning with peak constraints.

\section{Problem Formulation and Assumption}

We consider an episodic setting of the PCMDP with finite state and action space, defined by PCMDP $(\mathcal{S},\mathcal{A},H,\mathbb{P},r, f_i, s_1)$, where $\mathcal{S}$ is the state space with $\vert\mathcal{S}\vert=S$, $\mathcal{A}$ is the set of actions with $\vert\mathcal{A}\vert=A>1$, $H$ is the number of epochs in each episode, and $\mathbb{P}$ is the transition matrix such that $\mathbb{P}_h(\cdot\vert s,a)$ gives the probability distribution over next state based on the state and action pair $(s,a)$ at epoch $h$. Further, $r:\mathcal{S}\times\mathcal{A}\rightarrow\mathbb{R}$ is the deterministic reward function and $f_i:\mathcal{S}\times\mathcal{A}\rightarrow\mathbb{R}$,  $i=1, \cdots, I$, are peak constraint functions. $s_1$ is a fixed initial state. In the RL setting, the transition dynamics $\mathbb{P}_h$, the reward function $r$ and constraint functions $f_i$ are unknown to the agent but can be measured when a state action pair $(s,a)$ is observed. If we know the model of MDP (which means that the transition dynamics, the reward and constraint functions are known), we can solve the problem by solving the optimal Bellman Equation.
	\begin{equation}
	\tilde{Q}_h^{*}(s,a)=r_h(s,a)+[\mathbb{P}_h\tilde{V}_{h+1}^*(s,a)]\quad \tilde{V}_{h}^*(s)=\max_{a\in\mathcal{A}_h^{safe}(s)}\tilde{Q}_h^*(s,a)
	\end{equation}
	where $\tilde{Q}(s,a)$ is the state-action value function for reward function $r(s,a)$ such that 
	\begin{equation}
		\tilde{Q}_h^{\pi}(s,a)=\mathbf{E}[\sum_{h'=h+1}^{H}r_{h'}(s_{h'},\pi_{h'}s_{h'})|s_h=s,a_h=a]
	\end{equation} 
	and $\mathcal{A}_h^{safe}(s)=\{a: f_i(s,a)\geq 0,\forall i\in[I]\}$. As a result, the problem can be considered as an unconstrained MDP problem with specified action set for each $s\in\mathcal{S}$ and $h\in[H]$. In this paper, we make the following assumptions. 
\vspace{0.2in}
\begin{assumption}\label{ass_bounded}
	The absolute values of the reward function $r$ and constraint functions $f_i,i=1,\cdots,I$ are strictly bounded by a constant known to the agent. Without loss of generality, we let this constant be $1$. 
\end{assumption}
\vspace{0.2in}
\begin{assumption}\label{ass_pos}
	The values of the reward function $r$ is non-negative, i.e., $0\leq r(s,a)\leq 1,\forall(s,a)$. %
\end{assumption}
\vspace{0.2in}
These assumptions on reward function are typical in reinforcement learning \cite{jinchi,mengdi,Azar}, and the bound of reward function can be normalized. Further, the reward can be shifted up by adding a constant to make the reward function non-negative. 

\begin{remark}\label{remark_initial}(Nonidentical Initial State)
	Despite that the MDP model is defined with a fixed initial state for all episodes, the result in this paper still works for random initial state with a simple modification. Denote the distribution for the initial state as $d$ such that $s_1\sim d$. Then, we artificially add an extra state $s_0$ and define $r(s_0,a)=0,\forall a\in\mathcal{A}$, $f_i(s_0,a)=0,\forall a\in\mathcal{A},\forall i\in[I]$ and $\mathbb{P}_0(s_1\vert s_0,a)=d(s_1),\forall a\in\mathcal{A}$. Also, we modify $\mathbb{P}_h(s_0\vert s,a)=0,\forall (s,a)\in\mathcal{S}\times\mathcal{A}$. By this modification, $s_0$ is considered as a dummy state, which will not influence future epochs. Thus, the proposed algorithm can start from $s_0$ in this setting. Thus, a fixed initial state is without loss of generality. 
\end{remark}

We define the policy as a function that maps a state $s\in\mathcal{S}$ to a probability distribution of the actions with a probability assigned to each action $a\in\mathcal{A}$. In an episodic setting, the  policy $\bm{\pi}$ is a collection of $H$ policy functions $\pi_h$ at each epoch, that is $\pi_h(s)=a$ with probability $\mathbf{Pr}(a\vert s, h)$. Constrained RL problem is concerned with finding the optimal policy to achieve the highest total reward subject to a set of constraints, which can be formally stated as 

\begin{equation}\label{eq_formulation}
	\begin{aligned}
	\textbf{Original PCMDP}:\quad\max\limits_{\bm{\pi}}\quad & \mathbf{E}\bigg[\sum_{h=1}^{H}r(s_h,\pi_h(s_h))\bigg]\\
	\textup{s.t.}\quad &
	\mathbf{E}[f_i^-(s_h,\pi_h(s_h))]\geq 0~~ \forall h\in[H],\forall i\in[I]
	\end{aligned}
\end{equation}

where the expectation is taken with respect to  both the policy $\pi$ and the transition probability $\mathbb{P}_h$, and $x^-\triangleq\min\{x,0\}$. In the following parts, we use $\mathbf{E}$ instead of $\mathbf{E}_{\pi,\mathbb{P}}$ for simplicity, which is the expectation value on the randomness of policy and transition dynamics. The formulated problem in Eq. \eqref{eq_formulation} is called \textbf{Original PCMDP} in this paper, which optimizes the total reward and satisfies the peak constraints simultaneously.

\begin{remark}\label{remark_peak}
	Since this paper considers the episodic setting with a tabular MDP. For a fixed $h$, there exists a discrete distribution for each state action pair $(s_h,a_h)$ where $a_h\sim\pi_h(\cdot\vert s_h)$ and $s_h\sim\mathbb{P}_{h-1}(\cdot\vert s_{h-1},a_{h-1})$. Denote this distribution as $\lambda(s,a)$. Thus, the constraint in Eq. \eqref{eq_formulation} can be expressed as
	\begin{equation}
		\mathbf{E}[f_i^-(s_h,\pi_h(s_h))]=\sum_{s\in\mathcal{S},a\in\mathcal{A}} \lambda(s,a)f_i^-(s,a)
	\end{equation}
	If $f_i(s,a)<0$ with a positive probability, then $\sum_{s\in\mathcal{S},a\in\mathcal{A}} \lambda(s,a)f_i^-(s,a)<0$, which gives a contradiction with Eq. \eqref{eq_formulation}. 
	Thus, the constraint in \textbf{Original PCMDP} can be considered as $f_i(s_h,a_h)\geq 0$ with probability 1, equivalently.
\end{remark}

We emphasis that the proposed PCMDP problem is a special case of Constrained MDP problem mentioned in \cite{altman1999constrained}. The difference is that constraint functions need to be satisfied in each epoch $h$ in our formulation, while it is only needed to be satisfied on an average along one episode in \cite{altman1999constrained}. {  In the proof of Lemma \ref{lem_duality}, it can be seen that PCMDP can be converted to the standard CMDP with $HI$ constraints. A standard approach for constrained MDP would use $HI$ Q-tables, one for each constraint. However, in this work, we provide a low space-time complexity approach that uses a single Q-table.} Besides, it is well known that the optimal policy for the Constrained MDP with average constraint functions could be stochastic. However, the authors of \cite{ather} showed that there is a deterministic policy which is optimal for the Peak Constraint MDP.

In order to make the problem non-trivial, we assume that the problem has a feasible solution. More formally,

\begin{assumption}[Feasibility]\label{ass:feasibility}
	There exists a policy $\pi$ such that $\mathbf{E}[f_i^-(s_h,\pi_h(s_h))]\geq 0$ for all $h\in[H]$ and $ i\in[I]$. 
\end{assumption}

We will make use of all the three assumptions in the remainder of the paper.  

We note that based on the definition of $x^-=\min\{x,0\}$, Slater Condition \cite{lempio1974note} will not hold for Eq. \eqref{eq_formulation}. Thus, we introduce a new slack variable $\xi>0$ and define $g_{i,\xi}(s,a) = f_i^-(s,a)+\xi$ to formulate the \textbf{Approximated PCMDP} as follows.
\begin{equation}\label{eq_new_formulation}
	\begin{aligned}
		\textbf{Approximated PCMDP:}\quad \max\limits_{\bm{\pi}}\quad & \mathbf{E}\bigg[\sum_{h=1}^{H}r(s_h,\pi_h(s_h))\bigg]\\
		\textup{s.t.}\quad &\mathbf{E}[g_{i,\xi}(s_h,\pi_h(s_h))] \geq 0~~\forall h\in[H],~\forall i\in[I]
	\end{aligned}
\end{equation}

We notice that the introduction of approximation parameter $\xi>0$ relaxes constraints and makes the feasible region larger. In the next lemma, it is shown that the Slater Condition always holds for \textbf{Approximated PCMDP}.

\begin{lemma}\label{lem_slater}
	There exists a set of $\gamma_{h,i},h\in[H],i\in[I]$ satisfying $0<\gamma_{h,i}\leq \xi,\forall h,i$ and a policy $\pi$ such that 

	\begin{equation}
		\begin{aligned}
		\mathbf{E}[g_{i,\xi}(s_h,\pi_h(s_h))] \geq \gamma_{h,i},\forall h,i
		\end{aligned}
	\end{equation}

\end{lemma}

\begin{proof}
	By Assumption \ref{ass:feasibility}, we have 
	\begin{equation}
		\mathbf{E}[g_{i,\xi}(s_h,\pi_h(s_h))]=\mathbf{E}[f_i^-(s_h,\pi_h(s_h))]+\xi\geq \xi\geq \gamma_{h,i}
	\end{equation}
\end{proof}

We define the state value function $V_h^{\pi}:\mathcal{S}\rightarrow\mathbb{R}$ at epoch $h$ under policy $\pi$ as follows
\begin{equation}
	V_h^{\pi}(s):=\mathbf{E}\bigg[\sum_{h'=h}^{H}r(s_{h'},\pi_{h'}(s_{h'}))\vert s_h=s\bigg]\label{vhdefn}
\end{equation}
Denote the set $\Pi$ as the constraint set in which the policy satisfies the constraints in the Eq. \eqref{eq_formulation}.  We denote an optimal policy as $\pi^*$, which gives the optimal value function for the original problem as

\begin{equation}
V_h^*(s)=\sup\limits_{\pi\in\Pi}V_h^{\pi}(s),
\end{equation} 
for all $s\in\mathcal{S}$ and $h\in[H]$.  Note that the introduction of approximation parameter $\xi$ would change the optimal policy and thus the optimal value function for the relaxed problem is a function of $\xi$ which is denoted by $V_{1,\xi}^*(s_1)$. In this paper, we extend the concept of PAC to define a concept of $\epsilon$-optimal policy for both the reward function and constraint violation, which is given as follows.
\begin{definition}[$(\epsilon,p)$-PAC policy]
	For any $p\in(0,1)$, if a policy $\bar{\pi}$ satisfies the following equations with probability at least $1-p$, then we call it an $(\epsilon,p)$-PAC policy
	\begin{equation}\label{eq:eps_optimal}
		\begin{aligned}
		V_1^*(s_1)-V_1^{\bar{\pi}}(s_1)&\leq \epsilon\\
		\sum_{h=1}^{H}\sum_{i=1}^{I}\mathbf{E}\bigg|f_i^-(s_h,\bar{\pi}_h(s_h))\bigg|&\leq \epsilon
		\end{aligned}
	\end{equation}
\end{definition}
By this definition, an $(\epsilon,p)$-PAC policy is a policy for which the value function is $\epsilon$ close to the optimal and the total constraint violation is less than $\epsilon$. Moreover, due to the definition of $x^-$ and the absolute notation in Eq. \eqref{eq:eps_optimal}, we know that the peak constraint violation is less than $\epsilon$ for each epoch $h$. Notice that the $(\epsilon,p)$-PAC policy is defined with respect to \textbf{Original PCMDP}.

\section{Proposed Algorithm}

For any state action pair $(s,a)$, we define a modified reward function as
\begin{equation}\label{alt_defn_R}
	R_\xi(s,a)=r(s,a)+\frac{\eta}{I}\sum_{i=1}^{I} g_{i,\xi}^-(s,a)
\end{equation}
where $g_{i,\xi}^{-}:=\min\{g_{i,\xi},0\}$, $\eta=\frac{2HI}{\gamma}$ and $\gamma=\min_{h,i}\gamma_{h,i}$. This modified reward function gives nearly the original reward function $r(s,a)$ when all constraints function are satisfied because $f_i^-(s,a)=0$ in this case. Further, this provides a penalty function when any of the constraints is larger than $\xi$. 
Based on the modified reward function, we define a counterpart of the value function $W_{h,\xi}^{\pi}(s)$ as
\begin{equation}
	W_{h,\xi}^{\pi}(s):=\mathbf{E}\bigg[\sum_{h'=h}^{H}R_\xi(s_{h'},\pi_{h'}(s_{h'}))\vert s_h=s\bigg]
\end{equation}
Further, let $W_{1,\xi}^{*}(s_1) = \max_{\pi} W_{1,\xi}^{\pi}(s_1)$.  W with the notation $[\mathbb{P}_hV_{h+1}](s,a):=\mathbf{E}_{s'\sim\mathbb{P}_h(\cdot\vert s,a)}V_{h+1}(s')$, we define a counterpart of the state-action function $Q_h^{\pi}(s,a)$ as
\begin{equation}
	\begin{aligned}
		Q_{h,\xi}^{\pi}(s,a)&:=R_\xi(s,a)+\mathbf{E}\bigg[\sum_{h'=h+1}^{H}R_\xi\big(s_{h'},\pi_{h'}(s_{h'})\big)\bigg|s_h=s,a_h=a\bigg]=(R_\xi+\mathbb{P}_hW_{h+1,\xi}^{\pi})(s,a)
	\end{aligned}
\end{equation}
With these notations, we are able to define a modified unconstrained MDP problem, \textbf{Modified MDP}, as 
\begin{equation}\label{eq_modified_problem}
	\textbf{Modified MDP}:\quad \max\limits_{\pi}\quad \mathbf{E}\bigg[\sum_{h=1}^{H}R_\xi(s_h,\pi_h(s_h))\bigg]
\end{equation}

Recalling the assumption that the original reward function $r(\cdot)$ is bounded, we show the absolute value of the modified reward function $R_\xi$ is also bounded.
\vspace{0.2in}
\begin{lemma}\label{lem_bound_R}
	If $\gamma<\min\{\xi,2HI(1-\xi)\}$. the absolute value of the modified reward function $R(s,a)$ is bounded. Formally,
		\begin{equation}
		\vert R(s,a)\vert\leq \eta \quad \forall (s,a)\in\mathcal{S}\times\mathcal{A}
		\end{equation}
\end{lemma}
\begin{proof}
	If $f_i(s,a)\geq -\xi$, then $R_\xi(s,a)=r(s,a)$ and thus  $0\leq R_\xi(s,a)\leq 1$. 	
	Otherwise, we have from \eqref{alt_defn_R} that 
	\begin{equation}
		R_\xi(s,a)=r(s,a)+\frac{\eta}{I}\sum_{i=1}^{I} \big[f_i^-(s,a)+\xi]
	\end{equation}
	Notice that $-\eta\leq\frac{\eta}{I}\sum_{i=1}^{I}f_i^-(s,a)\leq 0$.
	Thus, we have $-\eta+\eta\xi\leq R_{\xi}(s,a)\leq 1+\eta\xi$. Since $\gamma<\min\{\xi,2HI(1-\xi)\}$, we have $-\eta\leq R_\xi(s,a)<\eta$. Then, the two cases together provide the result in the statement of the Lemma. 
\end{proof}
\begin{algorithm*}[tb]
	\caption{Constrained Q-Learning Algorithm}
	\label{alg:cons_q}
	\begin{algorithmic}[1]
		\State Initialize $Q_h(s,a)\leftarrow \eta H$, $W_h(s,a)\leftarrow \eta H$, $N_h(s,a)\leftarrow 0 $, $\mu_h(s,a)\leftarrow 0$, $\sigma_h(s,a)\leftarrow 0$ and $\beta_0(s,a,h)\leftarrow 0$ for all $(s,a,h)\in\mathcal{S}\times\mathcal{A}\times[H]$. Initial parameter $\xi$ and $\eta$
		\For{episode $k=1,...K$} 
		\State Observe $s_1$ 
		\For{step $h=1,...H$}
		\State Take action $a_h\leftarrow arg\max\limits_{a'}Q_h(s_h,a')$ and observe $s_{h+1}$
		\State $t=N_h(s_h,a_h)\leftarrow N_h(s_h,a_h)+1$
		\State $\mu_h(s_h,a_h)\leftarrow \mu_h(s_h,a_h)+W_{h+1}(s_{h+1})$
		\State $\sigma_h(s_h,a_h)\leftarrow \sigma_h(s_h,a_h)+(W_{h+1}(s_{h+1}))^2$
		\State $\beta_t(s_h,a_h,h)\leftarrow\min\{c_1(\sqrt{\frac{H}{t}(\frac{\sigma_h(s_h,a_h)-(\mu_h(s_h,a_h))^2}{t}+\eta H)\ell}+\eta\frac{\sqrt{H^7SA}\ell}{t}),c_2\eta\sqrt{\frac{H^3\ell}{t}}\}$
		\State $b_t\leftarrow \frac{\beta_t(s_h,a_h,h)-(1-\alpha_t)\beta_{t-1}(s_h,a_h,h)}{2\alpha_t}$
		\State $Q_h(s_h,a_h)\leftarrow(1-\alpha_t)Q_h(s_h,a_h)+\alpha_t[R_h(s_h,a_h)+W_{h+1}(s_{h+1})+b_t]$ ($R_h$ as defined in Eq. \eqref{alt_defn_R})
		\State $W_h(s_h)\leftarrow\min\{\eta H,\max\limits_{a'\in\mathcal{A}}Q_h(s_h,a')\}$
		\EndFor 
		\EndFor 
	\end{algorithmic}
\end{algorithm*}
We use the modified reward function to provide a Q-learning based algorithm as described in Algorithm \ref{alg:cons_q}. The basic steps of Q-learning follow from that in \cite{jinchi}, while are adapted to incorporate constraints.  In line 1, the agent initializes the Q-table and $N_h(s, a)$, which is the notation for the number of times that the state-action pair is taken at epoch $h$. In line 3, the agent is given an initial state at the beginning of each episode. Then, in line 5, the agent takes an action to maximize the current state-value function $Q_h(s_h, a_h)$ and observes the next state. $N_h(s,a)$ is updated in line 6. Line 7 to line 10 gives an efficient way to compute a Bernstein type UCB. Q-table and the W-table are then updated according to the line 11 and line 12, where $b_t$ is the upper confidence bound and $\alpha_t$ is the learning rate defined as $\alpha_t:=\frac{H+1}{H+t}$.

Using Algorithm \ref{alg:cons_q}, we find the policy $\pi_h^k$ at the step $h$ in episode $k$ is deterministic, and is defined as
\begin{equation}
	\pi_h^k(a\vert s_h^k) = \begin {cases}
	1 & a=\arg\min_{a'}Q_h^k(s_h^k,a')\\
	0 & \text{otherwise}\\
	\end{cases}
\end{equation}
Given a Markov Decision Problem with peak constraints, this paper shows that an $\epsilon$-optimal policy can be extracted from Algorithm \ref{alg:cons_q}. The guarantees of the proposed algorithm will be analyzed in the next section.

\section{PAC Analysis}
In this section, we will show that the proposed algorithm gives an $\epsilon$-optimal policy for $K$ large enough.   In order to prove the result, we first provide connections among the original problem \textbf{Original PCMDP}, the approximated problem \textbf{Approximated PCMDP} and modified unconstrained problem \textbf{Modified MDP}. Then, we derive the sub-linear result for \textbf{Modified MDP}. The main result can be derived by the guarantees for \textbf{Modified MDP} and its relation with \textbf{Original PCMDP} and \textbf{Approximated PCMDP}

Two following results, Lemma \ref{lem_opt_origin} and \ref{lem_opt_approx}, describe the relationship of optimal value function between \textbf{Modified MDP} and \textbf{Original PCMDP}/\textbf{Approximated PCMDP}, respectively.

\begin{lemma}\label{lem_opt_origin}
The optimal value function $V_1^*$ for \textbf{Original PCMDP} is  equal to the optimal value function $W_{1,\xi}^*$ for \textbf{Modified MDP}. More formally, 
	\begin{equation}
	V_{1}^*(s_1)=W_{1,\xi}^*(s_1)
	\end{equation}
\end{lemma}
\begin{proof}
	Considering the optimal policy $\pi^*$ in \textbf{Original PCMDP} in Eq. \eqref{eq_formulation}, we have 
	\begin{equation}\label{eq_opt1}
	V_1^*(s_1)=\mathbf{E}\bigg[\sum_{h=1}^{H}r(s_{h},\pi_{h}^*(s_{h}))\bigg]\overset{(a)}=\mathbf{E}\bigg[\sum_{h=1}^{H}R_\xi(s_{h},\pi_{h}^*(s_{h}))\bigg]\leq W_{1,\xi}^{*}(s_1)
	\end{equation}
	Step (a) holds because with feasible optimal policy in \textbf{Original PCMDP}, $f_{i}(s_h,a_h)\geq 0$ for any possible trajectories and thus
	$g_{i,\xi}(s_h,a_h)=f_i^{-}(s_h,a_h)+\xi=\xi$, which means
	\begin{equation}\label{eq_compare_R}
		R_\xi(s_h,a_h)=r(s_h,a_h)+\frac{\eta}{I}\sum_{i=1}^{I} \big[g_{i,\xi}^-(s_h,a_h)\big]=r(s_h,a_h)
	\end{equation}
	Moreover, the final inequality holds because the optimal policy for \textbf{Modified MDP} may be different from the \textbf{Original PCMDP} and any other policy will achieve less reward. For the other direction, consider the optimal policy $\pi^{W*}$ in the \textbf{Modified MDP} and follow the same step, 
	\begin{equation}
		W_{1,\xi}^{*}(s_1)=\mathbf{E}\bigg[\sum_{h=1}^{H}R_\xi(s_{h},\pi_{h}^{W*}(s_{h}))\bigg]\leq \mathbf{E}\bigg[\sum_{h=1}^{H}r(s_{h},\pi_{h}^{W*}(s_{h}))\bigg]\overset{(a)}\leq V_1^*(s_1)
	\end{equation}
	The first inequality holds by $g_{i,\xi}^-\leq 0$ in Eq. \eqref{eq_compare_R}. Thus, this gives the result as in the statement of the Lemma. 
\end{proof}

\begin{lemma}\label{lem_opt_approx}
	The optimal value function $V_{1,\xi}^*$ for \textbf{Approximated PCMDP} and the optimal value function $W_{1,\xi}^*$ for \textbf{Modified MDP} have the following relation:
	\begin{equation}
	 W_{1,\xi}^*(s_1)\le V_{1,\xi}^*(s_1)\le W_{1,\xi}^*(s_1)+\eta H\xi
	\end{equation}
\end{lemma}
\begin{proof}
	Considering the optimal policy $\pi^*$ in \textbf{Approximated PCMDP} in Eq. \eqref{eq_new_formulation}, define a distribution $d_h^{\pi^*,\mathbb{P}}$ such that $(s_h,a_h)\sim d_h^{\pi,\mathbb{P}}$ with policy $\pi^*$ and transition dynamics $\mathbb{P}$, then
	\begin{equation}
	\begin{aligned}
	\mathbf{E}[g_{i,\xi}^{-}(s_h,\pi_h^*(s_h))]&=\sum_{s,a}g_{i,\xi}^{-}(s,a)d_h^{\pi^*,\mathbb{P}}(s,a)\\
	&=\sum_{g_{i,\xi}(s,a)>0}g_{i,\xi}^{-}(s,a)d_h^{\pi^*,\mathbb{P}}(s,a)+\sum_{g_{i,\xi}(s,a)<0}g_{i,\xi}^{-}(s,a)d_h^{\pi^*,\mathbb{P}}(s,a)\\
	&=\sum_{g_{i,\xi}(s,a)<0}g_{i,\xi}(s,a)d_h^{\pi^*,\mathbb{P}}(s,a)\\
	&=\mathbf{E}[g_{i,\xi}(s_h,\pi_h^*(s_h))]-\sum_{g_{i,\xi}(s,a)>0}g_{i,\xi}(s,a)d_h^{\pi^*,\mathbb{P}}(s,a)\geq-\xi
	\end{aligned}
	\end{equation} 
	where the last step holds because $\mathbf{E}[g_{i,\xi}(s_h,\pi_h(s_h))]\geq 0$ by the formulation \eqref{eq_new_formulation} and $g_{i,\xi}\leq \xi$ by the definition. Thus,
	\begin{equation}
	\begin{aligned}
	\mathbf{E}[R_\xi(s_{h},\pi_{h}^*(s_{h}))]&=\mathbf{E}[r(s_{h},\pi_{h}^*(s_{h}))]+\frac{\eta}{I}\sum_{i=1}^{I} \mathbf{E}\big[g_{i,\xi}^-(s_{h},\pi_{h}^*(s_{h}))\big]\geq \mathbf{E}[r(s_{h},\pi_{h}^*(s_{h}))]-\eta\xi
	\end{aligned}	
	\end{equation}
	Finally, by the definition of value function in \eqref{vhdefn}, we have
	\begin{equation}\label{eq_opt2}
	\begin{aligned}
	V_{1,\xi}^*(s_1)&=\mathbf{E}\bigg[\sum_{h=1}^{H}r(s_{h},\pi_{h}^*(s_{h}))\bigg]\leq \mathbf{E}\bigg[\sum_{h=1}^{H}R_\xi(s_{h},\pi_{h}^*(s_{h}))\bigg]+\eta H\xi\leq W_{1,\xi}^{*}(s_1)+\eta H\xi
	\end{aligned}
	\end{equation}
	where the final inequality holds because the optimal policy for \textbf{Modified MDP} may be different from the \textbf{Approximated PCMDP} and any other policy will achieve less reward. For the other direction, recall that \textbf{Approximated PCMDP} is a relaxed version of \textbf{Original PCMDP} and thus
	\begin{equation}
		V_{1,\xi}^*(s_1)\geq V_{1}^*(s_1)=W_{1,\xi}^*(s_1)
	\end{equation}
\end{proof}

The following lemma describes the relationship of value function with a certain policy between \textbf{Modified MDP} and \textbf{Original PCMDP}.

\begin{lemma}\label{lem_policy}
	The value function with policy $\pi$ in \textbf{Modified MDP}, $W_{1,\xi}^{\pi}$, can be expressed by the value function with policy $\pi$ in \textbf{Original PCMDP}, $V_1^{\pi}$,  with a gap term describing the violation of constraints. More formally,
	\begin{equation}\label{eq_policy}
	W_{1,\xi}^{\pi}(s_1)=V_{1}^{\pi}(s_1)+\frac{\eta}{I}\sum_{h=1}^{H}\sum_{i=1}^{I}\mathbf{E}\big[g_{i,\xi}^-(s_h,\pi_h(s_h))\big]
	\end{equation}
\end{lemma}

\begin{proof}
	According to the definition of function $W$, we expand it as follows.
	\begin{equation}
	\begin{aligned}
		&W_{1,\xi}^{\pi_k}(s_1^k)=\mathbf{E}\bigg[\sum_{h=1}^{H}R_\xi(s_h^k,\pi_h^k(s_h^k))\bigg]\\
		&=\mathbf{E}\bigg[\sum_{h=1}^{H}r(s_h^k,\pi_h^k(s_h^k))\bigg]+\frac{\eta}{I}\sum_{h=1}^{H}\sum_{i=1}^{I}\mathbf{E}\big[g_{i,\xi}^-(s_h^k,\pi_h^k(s_h^k))\big]\\
		&=V_{1,\xi}^{\pi_k}(s_1^k)+\frac{\eta}{I}\sum_{h=1}^{H}\sum_{i=1}^{I}\mathbf{E}\big[g_{i,\xi}^-(s_h^k,\pi_h^k(s_h^k))\big]
	\end{aligned}		
	\end{equation}
	which is the result as in the statement of the Lemma.
\end{proof}	

The following lemma gives the sub-linear regret for \textbf{Modified MDP}.
\vspace{0.2in}
\begin{lemma}\label{lem_sublienar}
	For any $p\in(0,1)$, let $\ell=\log(SAT/p)$, {where $T=KH$}. Then,  for $K\geq H^5S^2A^2\ell^3$, the bound on the regret for \textbf{Modified MDP} with Algorithm \ref{alg:cons_q} is given as
	\begin{equation}
	\sum_{k=1}^{K}[W_{1,\xi}^*(s_1)-W_{1,\xi}^{\pi^k}(s_1)]\leq O(\eta \sqrt{H^3SAT\ell})
	\end{equation}
	with probability at least  $1-p$.
\end{lemma}
\begin{proof}
	The proof follows  from Theorem 2 in \citep{jinchi},  noting that the modified reward is bounded by $\eta$ and not $1$ as in \citep{jinchi}.
\end{proof}

Combining results of the above lemmas, the next theorem shows that the proposed algorithm is $\epsilon$-optimal.

\begin{theorem}\label{thm_bound}
For any $p\in(0,1)$, $H\ge 2$,  and any $\epsilon>0$, take $K=\Omega(\frac{I^2H^6SA\ell}{\epsilon^2})$, $\xi = \frac{\epsilon}{6HI}$, and $\ell=\log(SAT/p)$, {where $T=KH$}. Further, define $\bar{\pi}$ as a policy which is given as
	\begin{equation}
		\bar{\pi}(s) = \begin {cases}
		\pi^1(s) & \text{ with probability } 1/K\\
		\cdots & \cdots\\
		\pi^k(s) & \text{ with probability } 1/K\\
		\cdots & \cdots\\
		\pi^K(s) & \text{ with probability } 1/K
		\end{cases}
	\end{equation}
	Note that $\bar{\pi}$ chooses the different policies $\pi^k$ for $k\in[K]$ uniformly at random. If $H\geq 2$, then, the policy $\bar{\pi}$ is an $(\epsilon,p)$-PAC policy.
\end{theorem}

\begin{proof}
	The detailed proof is provided in the Appendix \ref{app_thm1}. 
\end{proof}
\begin{remark}
	The above result shows that we obtain an {$(\epsilon,p)$-PAC policy} if the number of episodes satisfies  $K\geq\Omega(\frac{I^2H^6SA\ell}{\epsilon^2})$. Thus, $K\geq \Omega(1/\epsilon^2)$, which is the best result for the problem even without constraints \cite{jinchi}. 
	
	We also note that even though the Theorem statement mentions $H\ge 2$, the result also holds when $H=1$, where we obtain $2\epsilon$-optimal strategy. Thus, the {$(\epsilon,p)$-PAC policy} will hold by replacing $\epsilon$ with $\epsilon/2$. 
\end{remark}

In the next theorem, it shows that the deterministic policy $\pi^k$ also converges to optimal value and constraint violations converge to zero. However, there exists a trade-off of the convergence rate between regret and constraint violation, which is slower than the uniform policy proposed in theorem \ref{thm_bound}.
\begin{theorem}\label{cor_deterministic}
By choosing $\eta=T^{\phi}$ {where $T=KH$},  $\phi\in(0,\frac{1}{2})$, and setting $\xi=0$, {the total regret and constraint violation with policy $\pi^k$ in Algorithm \ref{alg:cons_q} are bounded as}
	\begin{equation}
		\begin{aligned}
		\sum_{k=1}^{K}[V_{1}^*(s_1)-V_1^{\pi^k}(s_1)]&\leq O(T^{\frac{1}{2}+\phi}\sqrt{H^3SA\ell})\\
		\sum_{k=1}^{K}\sum_{h=1}^{H}\sum_{i=1}^{I}\mathbf{E}\big|g_i^-(s_h^k,\pi_h^k(s_h^k))\big|&\leq O(IT^{1-\phi})
		\end{aligned}
	\end{equation}
\end{theorem}
\begin{proof}
	The detailed proof is provided in the Appendix \ref{app_slow_convergence}. 
\end{proof}
\begin{remark}
	Despite that Theorem 2 shows deterministic policy also achieves sublinear regret and constraint violation, it doesn't achieve $\mathcal{O}(\sqrt{T})$ regret for both the objective and constraint violations. Whether a deterministic policy achieves optimal regret is left as a future direction.
\end{remark}

\section{Simulations}
\subsection{Energy Harvesting Communication System}
\begin{figure}[ht]
	\vskip -0.1in
	\begin{center}
		\centerline{\includegraphics[width=\columnwidth]{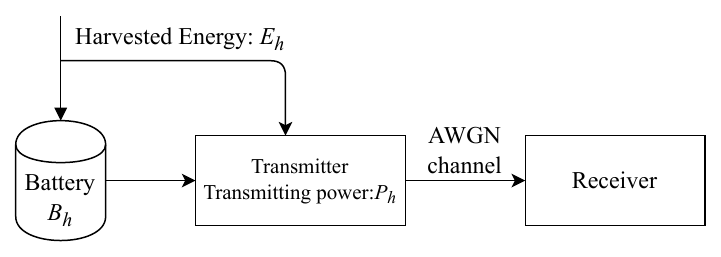}}
		\caption{Energy Harvesting Communication System}
		\label{system}
	\end{center}
	\vskip -0.2in
\end{figure}
In this section, we evaluate the proposed algorithm on a communication channel, where the transmitter is powered by renewable energy. Such a model has been studied widely in communication systems \cite{tutuncuoglu2015optimum,Blasco,Yang,wang2015paper,wang2014power}. In this model, we assume that the time is divided into time-slots. As shown in Fig. \ref{system}, in each time-slot, the transmitter can send data over an Additive Gaussian White Noise (AWGN) channel, where the signal transmitted by the transmitter gets added by a noise given by complex normal with zero mean and unit variance $\mathbb{CN}(0,1)$ at each time instance within the time-slot. We assume that the transmitter can use a power of $P_h$ in time-slot $h$, where the transmission is limited by a maximum power of $\bar{P}$. 

We assume that the transmitter is powered by a renewable energy source, where energy $E_h$ arrives during time-slot $h-1$ and can be used for time-slot $h$. Further, the transmitter is attached to a battery, which has a capacity of $\bar{B}$. The transmitter can use the energy from the existing battery capacity at the start of time-slot $h$, $B_h$, or the new energy arrival $E_h$. The energy from $E_h$ that is not utilized is stored in the battery. Thus, the battery state evolves as 
\begin{equation}\label{battery}
	B_{h+1}=\min\{\bar{B}, B_h+E_h-P_h\}. 
\end{equation}
We wish to optimize an upper bound on the reliable transmission rate \cite{wang2015paper}, given as 
$\mathcal{C} = \sum_h \log(1+P_h). $
We note that the transmission constraints can be modeled as peak constraints. Thus, the overall optimization problem is given as 
\begin{align}
	\max\limits_{P_h, h=1, \cdots, H}&\quad\mathbb{E}\bigg[\sum_{h=1}^{H}\log(1+P_h)\bigg]\label{opt_pro}\\
	s.t.& \quad B_{h+1}=\min\{\bar{B}, B_h+E_h-P_h\}\label{state_evolve}\\
	& \quad B_h\geq 0\label{action_space}\\
	&\quad P_h\leq \bar{P} \quad w.p.1\label{peak_constrain}
\end{align} 
We note that the expectation in the above is over the energy arrivals $E_h$, which makes the choice of $P_h$ stochastic. If the energy arrivals $E_h$ are known non-causally (known at $h=1$ for the entire future), the problem is convex and can be solved efficiently using the dynamic water-filling algorithm proposed in \cite{wang2015paper} or dynamic programming based solutions \cite{Blasco,wang2014power}. However, in realistic systems, $E_h$ is only known at time-slot $h$.

We will now model the problem as an MDP. The state at time-slot $h$ is given as $S_h=(B_h, E_h)$, which are the current battery level and the energy arrival. The energy $E_h$ is known causally, and the distribution is unknown. The action is the transmission power $P_h$. Eq. \eqref{state_evolve} gives the state evolution and the $E_h$ may evolve based on some Markov process in general. Eq. \eqref{action_space} restricts battery level must be positive and indirectly gives the action space for $P_h$ such that $\mathcal{A}=\{0,1,...,B_h+E_h\}$. Finally, the objective is given in \eqref{opt_pro} and the peak power constraint is given in Eq. \eqref{peak_constrain}. Notice that the the constraint function is not known in advance.

We set the distribution of $E_h$ as truncated Gaussian with mean $\mu$ and standard deviation $\sigma$, where the truncation levels are $0$ and $\bar{E}$, and we let it be independent across episodes. The problem is discretized to integers in order to apply the proposed algorithm. According to the selection of the parameters in \cite{wang2014power}, we set the horizon $H=20$ time-slots, battery capacity $\bar{B}=20$, power constraint $\bar{P}=8$, maximal harvest energy $\bar{E}=20$, mean and standard deviation $\mu=10$, $\sigma=5$, respectively. 

In the simulation, 1000 trajectories are generated by the above MDP. In Fig. \ref{fig:Learn_pro}, we plot the mean and variance for the sum of the transmission rate $V_1^{\bar{\pi}^k}$ and constraint violations defined in Eq. \eqref{eq:eps_optimal} and compare the learning speed for different choice of $\xi=0.1,0.01,0.001$. The Slater parameter for $\gamma$ is chosen to be $\frac{\xi}{2}$. Note that in each episode $k$, we evaluate the policy $\bar{\pi}^k$, which is the `averaged' policy by $\pi^1,\pi^2,\cdots, \pi^k$ as defined in Theorem \ref{thm_bound}. It can be seen that the sum of transmission rate converges to the optimal and the constraint violation converge to 0 as the number of episodes increases. Moreover, we can find that larger $\gamma$ gives a higher convergence speed while the difference between convergence speed for different $\gamma$ is quite small.

In order to compare the proposed algorithm, we consider three other baseline algorithms: the greedy policy, the balanced policy, and the optimal non-causal algorithm. The greedy policy tries to consume the harvested energy as much as possible in each slot, as calculated by $P_h = \min(\bar{P}, B_h +E_h)$. We also consider a balanced policy that consumes the fixed amount of energy in each slot if available, where the fixed value is calculated by $\sum_{h=1}^H E_h/H$, while that is limited by the available energy at each time. However, the balanced algorithm uses the future energy arrivals and is not a causal strategy. Further, the optimal strategy when all future energy arrivals are non-causally known is also used to show the performance of the proposed algorithm. We note that the proposed algorithm only assumes that the constraint function in state $s$ and action $a$ can be queried, but the function is not explicitly known, thus, the algorithms that project to the constraint function are not considered as they require complete knowledge of the function.

\begin{figure}
	\centering
	\begin{minipage}{.47\textwidth}
		\centering
		\includegraphics[width=\textwidth]{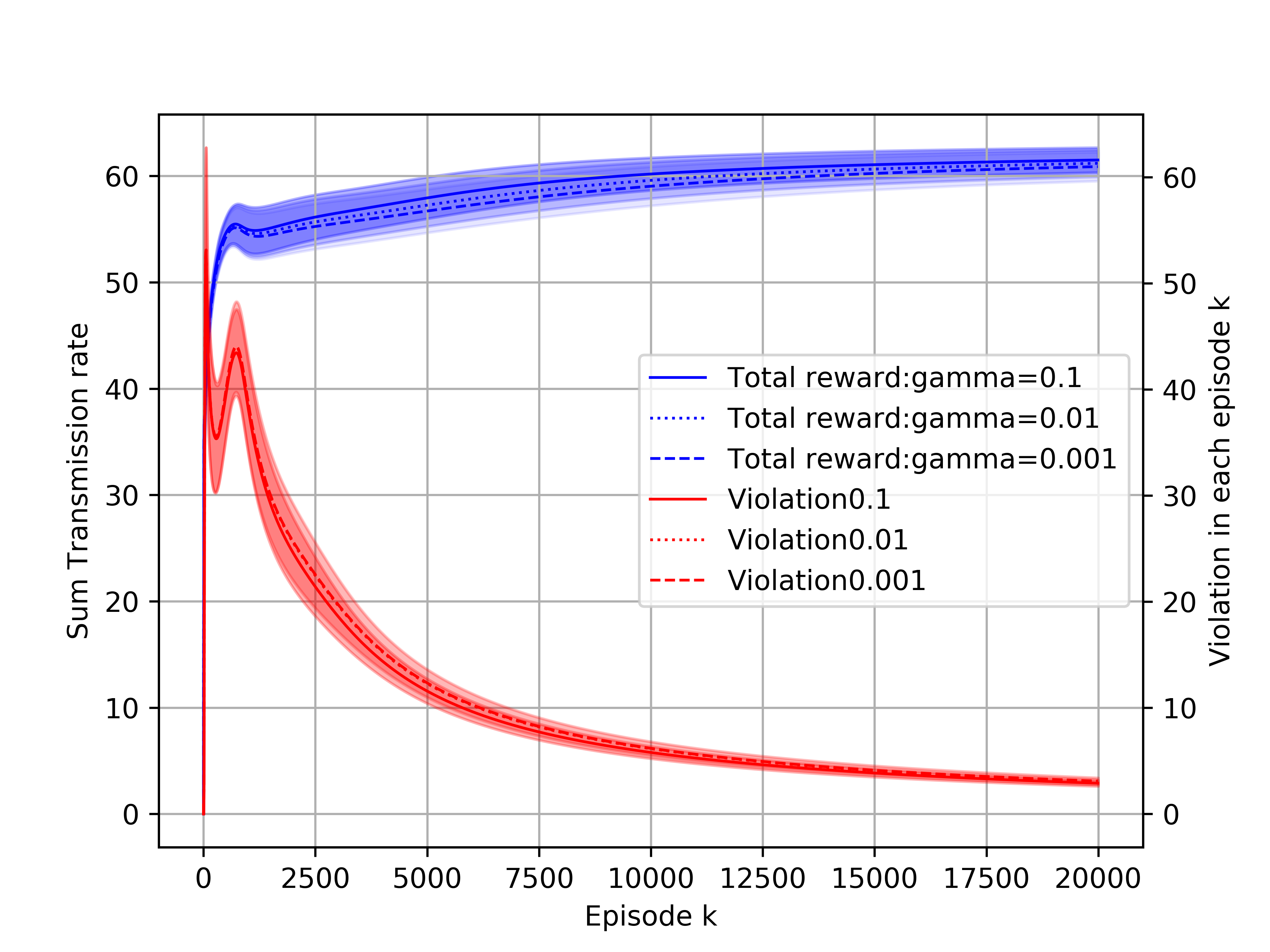}
		\caption{Convergence of Reward Function and Constraint Violations for the Proposed Algorithm}
		\label{fig:Learn_pro}
	\end{minipage}%
	\hspace{.2in}
	\begin{minipage}{.47\textwidth}
		\centering
		\includegraphics[width=\textwidth]{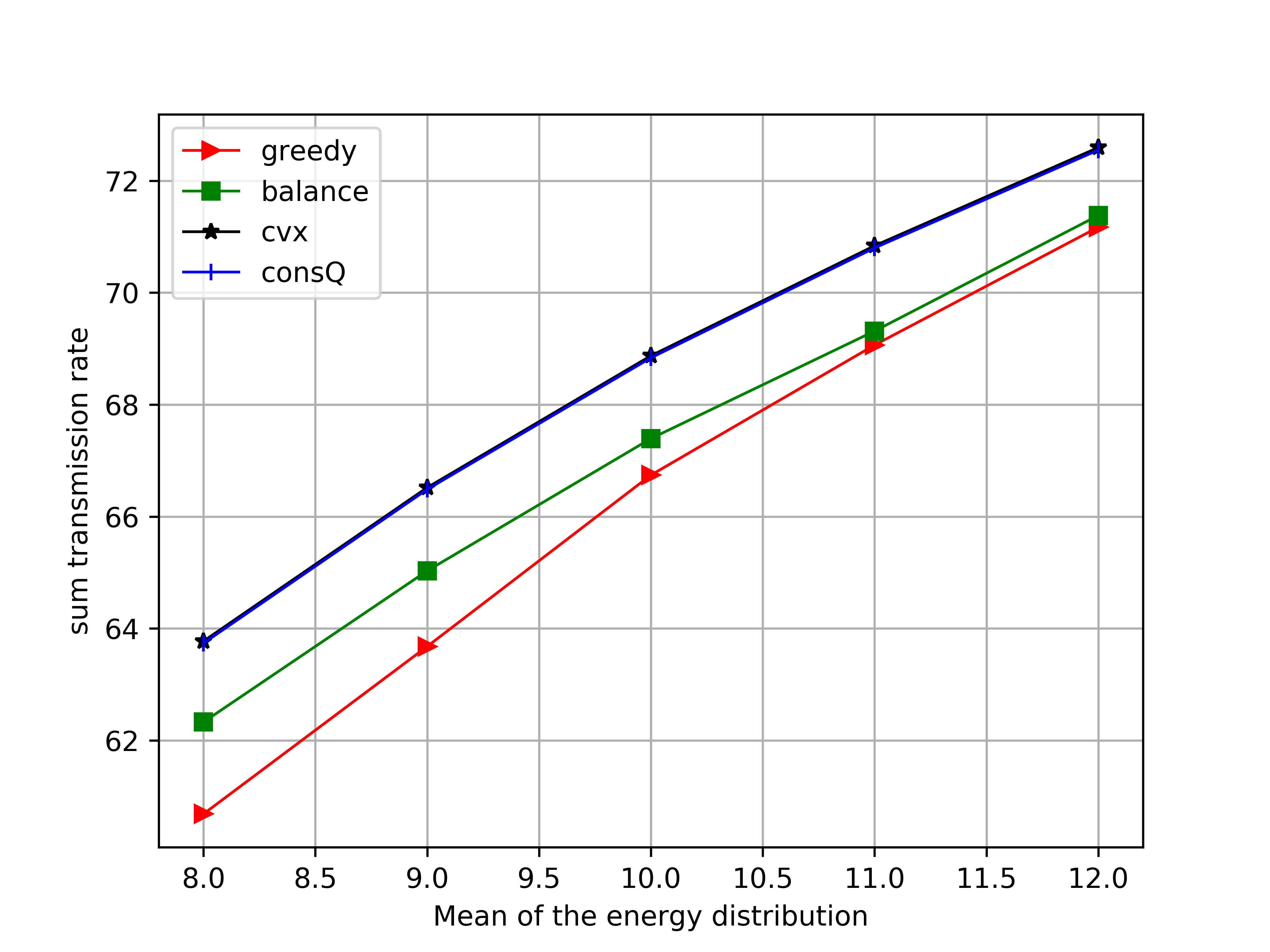}
		\caption{Performance Comparison Between Different Algorithms}
		\label{fig:compare}
	\end{minipage}
\end{figure}

In Fig. \ref{fig:compare}, we set the mean of the harvested energy as $8,9,10,11,12$ and compare the performance of different algorithms. In order to illustrate the policy convergence to the optimal, we choose $K=50000$ for the proposed algorithm. $\bar{P}$ is set to $15$ in this comparison. We see the balanced policy achieves higher performance than greedy policy because the energy can be allocated more reasonably while requiring non-causal information of energy arrival. The performance of the non-causal convex solver achieves the highest reward since it is an upper bound on the performance. However, the proposed algorithm achieves nearly the same performance as the upper bound, which shows that our algorithm is able to achieve the optimal solution. Furthermore, the proposed algorithm doesn't need any prior knowledge of the harvested energy and the constraint functions, which is a great advantage over the convex solver.

\subsection{Single machine scheduling problem with deadlines}
In this subsection, we evaluate our proposed algorithm on another problem, single machine scheduling with deadline. We adopt the problem formulated in \cite{KOULAMAS2001447}. In this scheduling problem, we assume there are a total of  $n$ jobs and each job is released at the beginning. Each job $i$ has a process time $p_i$, due time $d_i$, and deadline $\bar{d}_i$. Notice that the deadline is different from due time in the sense that the deadline cannot be violated. The tardiness of each job is computed by $T_i=\max\{0,C_i-d_i\}$, where $C_i$ is the completion time of job $i$. The preemption is not allowed in this problem. We want to minimize the maximal tardiness $T_{max}=\max_{i\in N}T_i$ and satisfy each deadline at the same time. This problem is  written as $1|\bar{d}_i| T_{max}$ in the standard scheduling notation. If the information of this problem, the process time $p_i$, the due time $d_i$, and deadline $\bar{d}_i$ are given in advance, there exists an optimal algorithm DT proposed in \cite{KOULAMAS2001447} and thus the DT algorithm can be considered as an offline baseline.

The scheduling problem with the deadline which cannot be violated is an important setting in the field of scheduling. This problem has several applications in practice such as the chemical industry \cite{KOULAMAS2001447}, energy efficient packet transmission \cite{packet}, workflow scheduling \cite{workflow}, etc. In order to solve this problem with the proposed algorithm, it needs to be first formulated as a PCMDP.

The state of CMDP $s$ in this problem consists of three parts, time $t$, job states $js$, and maximal  tardiness $maxT$ such that $s=[t, js, maxT]$. The job states $js$ are made of state of each job such that $js=[js_1,js_2,\cdots,js_N]$ where $js_i=0$ means the job $i$ has not been started and $js_i=1$ means the job $i$ has been completed. $maxT$ here stands for the current maximal tardiness which can be defined as $maxT=\max_{i\in A_t}T_i$ and $A_t$ is the set including all completed jobs until time $t$. The action space of CMDP is the list of uncompleted jobs. Denote $a$ as the action, $t'$, $js'$, $maxT'$ as the state of CMDP in next time. The dynamics of this PCMDP is written as follow
\begin{equation}
	\begin{aligned}
	t'&=t + p_a\\
	js'&=[js_1,\cdots,1,\cdots, js_N](\text{where 1 is at position a})\\
	maxT'&=\max\{maxT, \max\{0, t + p_a - d_a\}\}
	\end{aligned}
\end{equation} 

Besides, according to the objective of problem setting which is to minimize the maximum of tardiness, the reward function and constraint function are designed as 
\begin{equation}
	\begin{aligned}
	r(s,a)&=-\max\{0, t + p_a - d_a - maxT\}\\
	f(s,a)&=\max\{0, t + p_a -\bar{d}_a\}
	\end{aligned}
\end{equation}
It can be seen that the dynamics, reward, and constraint functions are totally decided by current state and action. Thus, the above formulation can be considered as a CMDP.

\begin{table}
	\centering
	\caption{Information of jobs in the 1st example of Single Machine Scheduling}
	\label{example1}
	\begin{tabular}{cccccc} 
		\hline
		Processing time ($p_i$) & 3 & 5 & 7 & 9 & 10 \\ 
		\hline
		Due Time ($d_i$) & 22 & 30 & 33 & 15 & 18 \\ 
		\hline
		Deadline ($\bar{d}_i$) & 30 & 28 &35 & 18 & 21 \\ 
		\hline	
	\end{tabular}
\end{table}

\begin{table}
	\centering
	\caption{Information of jobs in 2nd example of Single Machine Scheduling}
	\label{example2}
	\begin{tabular}{cccccccccc} 
		\hline
		Processing time ($p_i$) & 2 & 3 & 5 & 8 & 13 & 21 & 34 & 17 & 19\\ 
		\hline
		Due Time ($d_i$) & 75 & 70 & 65 & 60 & 88 & 35 & 59 & 100 & 100 \\ 
		\hline
		Deadline ($\bar{d}_i$)  & 70 & 70 & 70 & 100 & 90 & 40 & 60 & 130 & 110 \\ 
		\hline	
	\end{tabular}
\end{table}

{ 
\begin{table}
	\centering
	\caption{Information of jobs in the 3rd example of Single Machine Scheduling}
	\label{example3}
	\begin{tabular}{cccccc} 
		\hline
		Processing time ($p_i$) & $\sim U(2,4)$ & $\sim U(4,6)$ & $\sim U(3,8)$ & $\sim U(8,11)$ & $\sim U(8,11)$ \\ 
		\hline
		Due Time ($d_i$) & 22 & 30 & 33 & 15 & 12 \\ 
		\hline
		Deadline ($\bar{d}_i$) & 30 & 28 & 35 & 18 & 23 \\ 
		\hline	
	\end{tabular}
\end{table}
}

To test our algorithm and compare with the offline DT algorithm, we manually design two setups with 5 jobs and 9 jobs, respectively.  {  With 5 jobs, we consider two scenarios, where the processing time of the jobs are deterministic or stochastic. For 9 jobs, we consider deterministic processing times. We label the three examples here as Example 1: 5 jobs, deterministic processing time, Example 2: 9 jobs, deterministic processing time, and Example 3: 5 jobs, stochastic processing time. The details for these examples are listed in  Tables \ref{example1}, \ref{example2} and \ref{example3}, respectively.}

The comparison among the proposed algorithm, offline DT algorithm and Earliest Deadline Date (EDD) algorithm for three examples are shown in {  Fig. \ref{fig:compare1}, \ref{fig:compare2}, and \ref{fig:compare3}, respectively}. { In Fig. \ref{fig:compare3}, we run 100 independent tests and plot the average maximal tardiness and the standard variance.} It can be seen that EDD algorithm always selects the job which has the earliest deadline at each time.Further, EDD algorithm in this setting is sub-optimal. Furthermore, the total reward of the proposed algorithm converges to optimal DT value and the constraint violation converges to 0, which matches the theoretical result. Notice that the proposed algorithm is an online algorithm and thus doesn't need any information of job in advance, which is a great advantage over the offline DT algorithm.

\begin{figure}
	\centering
	\begin{minipage}{.47\textwidth}
		\centering
		\includegraphics[width=\textwidth]{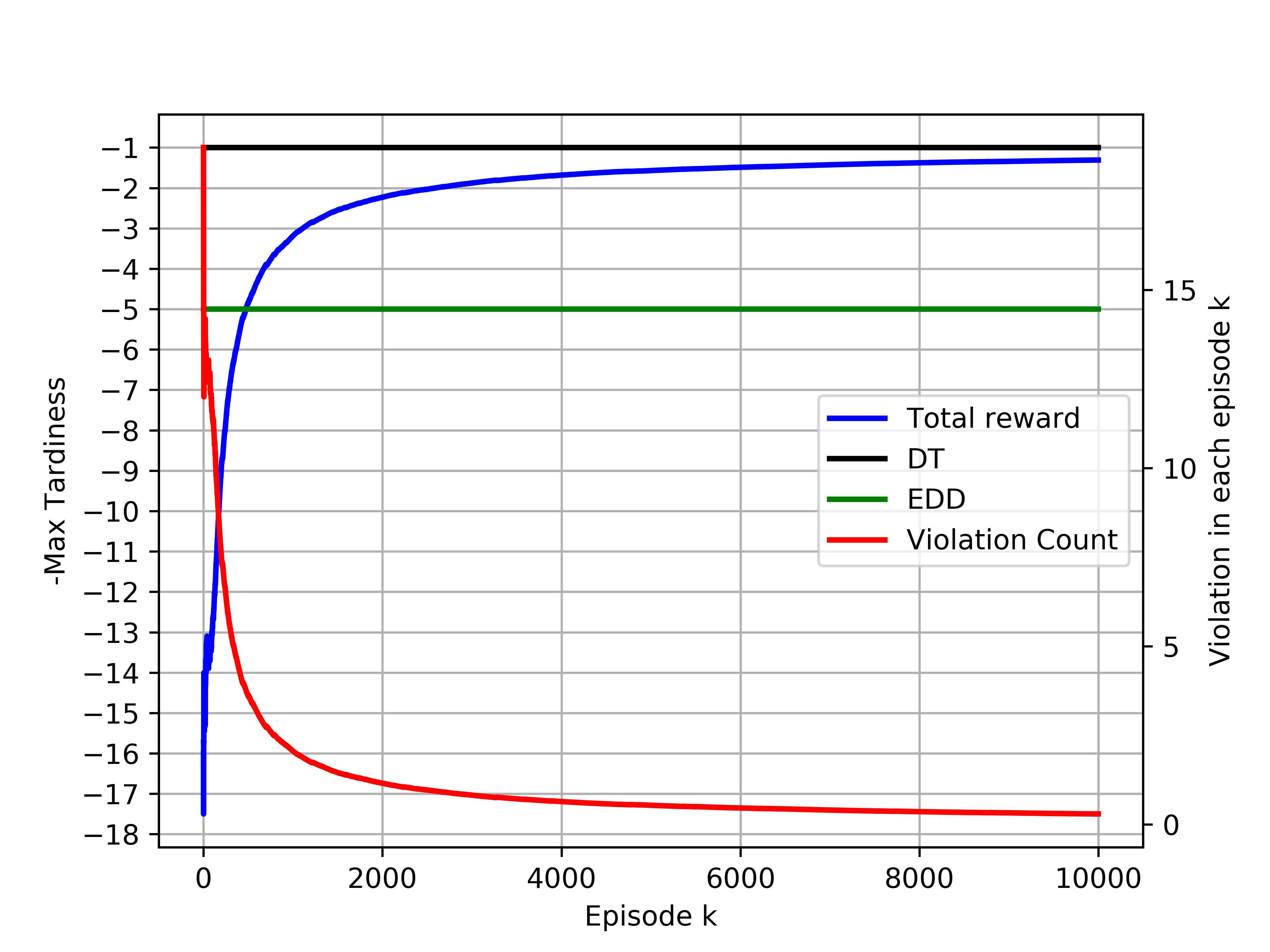}
		\caption{Comparison of proposed algorithm with DT and EDD in 1st example}
		\label{fig:compare1}
	\end{minipage}%
	\hspace{.2in}
	\begin{minipage}{.47\textwidth}
		\centering
		\includegraphics[width=\textwidth]{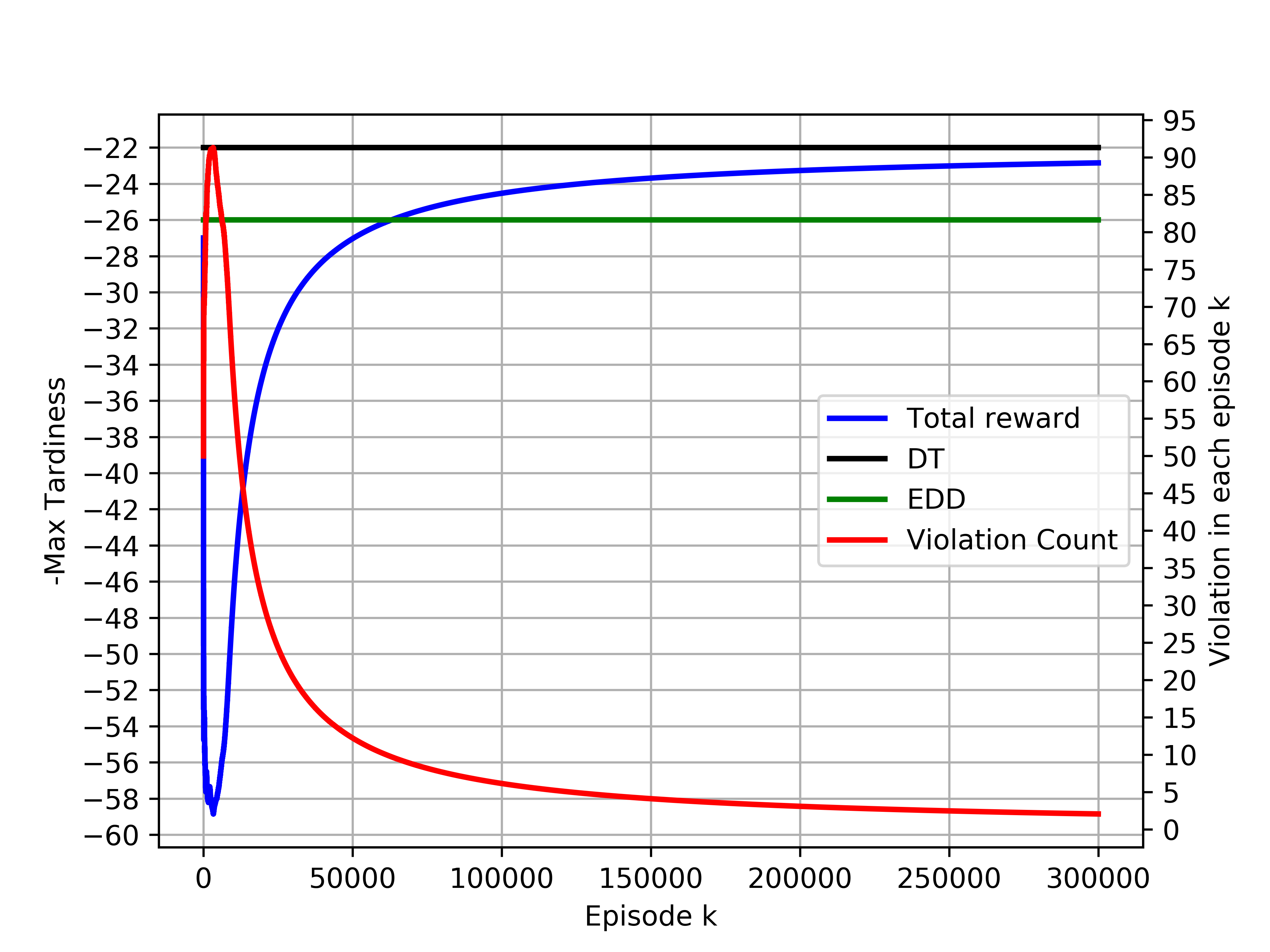}
		\caption{Comparison of proposed algorithm with DT and EDD in 2nd example}
		\label{fig:compare2}
	\end{minipage}
\end{figure}

{ 
\begin{figure}
	\centering
	\includegraphics[width=0.7\textwidth]{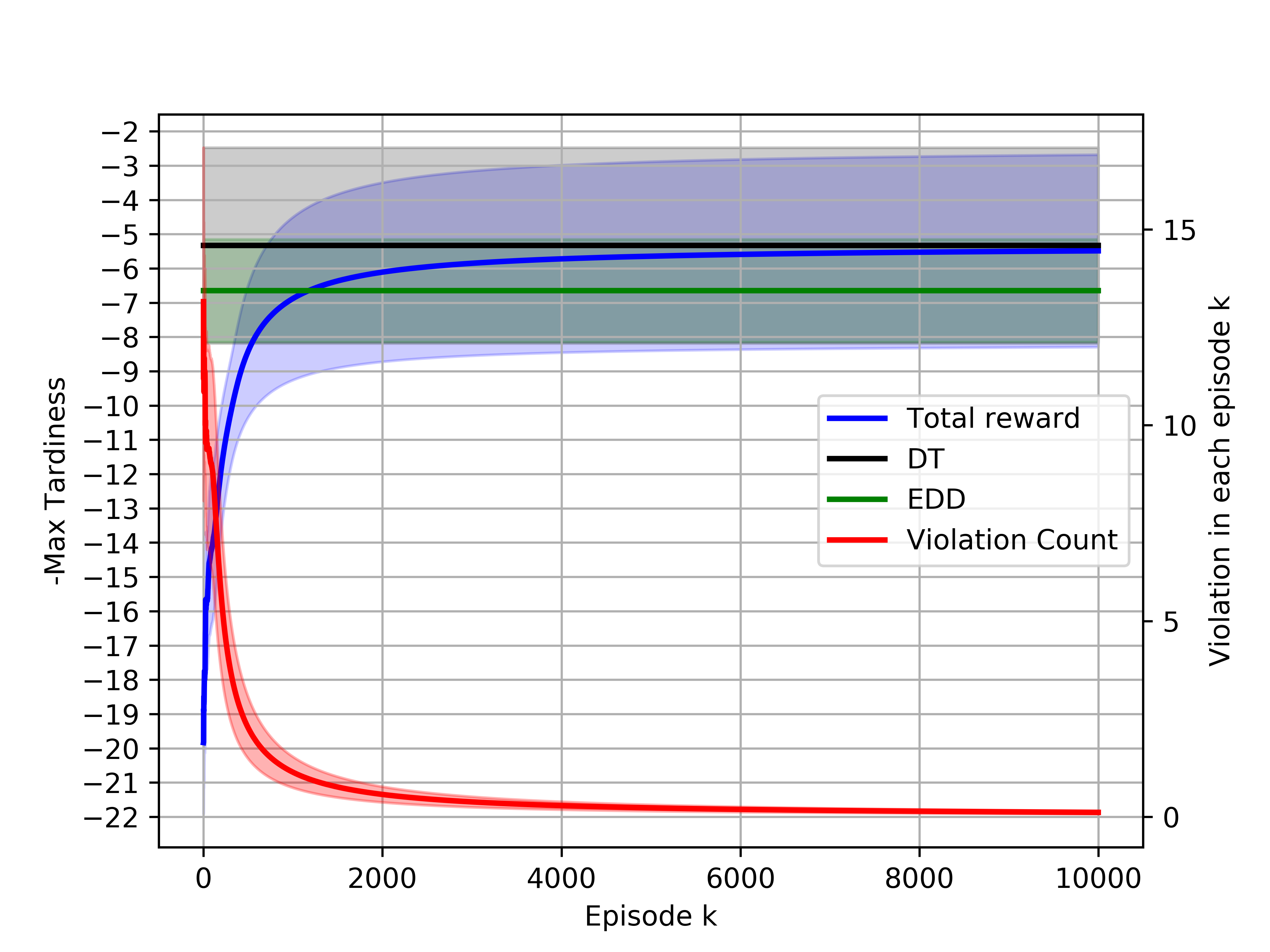}
	\caption{Performance comparison between different algorithms in 3rd example}
	\label{fig:compare3}
\end{figure}
}
Finally, it is noteworthy to see that the proposed algorithm can also be applied in the other settings with deadline such as the total tardiness \cite{KOULAMAS2001447}, total weight late work \cite{deadline2019} which has been shown as NP-hard and lack an optimal algorithm. However, the proposed algorithm can achieve $\epsilon$-optimal in finite time.

\section{Conclusion}
In this paper, we formulate a constrained MDP problem with a set of peak constraints. By using a modified reward function, we convert the problem into an unconstrained RL problem and propose a novel model-free algorithm. This paper proves our algorithm can achieve  $\epsilon$-optimal objective and constraint violations. We note this is the first result of PAC analysis for PCMDP when the state evolution and the constraint functions are unknown. The result is applied to an energy harvesting communication link and a single machine scheduling with deadline, and the proposed algorithm is shown to be close to the non-causal optimal solution.

\appendix
\section{Proof of Theorem \ref{thm_bound}}\label{app_thm1}
\subsection{Bounding the Regret}
	Combining  the results in Lemmas \ref{lem_opt_origin},   \ref{lem_policy}, and \ref{lem_sublienar}, for large enough $T$, we have
	\begin{equation}\label{eq_bound_regret}
		\sum_{k=1}^{K}[V_{1}^*(s_1)-V_1^{\pi^k}(s_1)]-\frac{\eta}{I}\sum_{k=1}^{K}\sum_{h=1}^{H}\sum_{i=1}^{I}\mathbf{E}\big[g_{i,\xi}^-(s_h^k,\pi_h^k(s_h^k))\big]\leq \sum_{k=1}^{K}W_{1,\xi}^*(s_1)-W_{1,\xi}^{\pi^k}(s_1)\leq O(\eta\sqrt{H^3SATl})	
	\end{equation}
	We note that the sum of terms $g_{i,\xi}^-$ (Accumulation of constraint violations) on the left hand side of \eqref{eq_bound_regret} is non-positive, which gives
	\begin{equation}\label{eq_bound_regret2}
		\frac{1}{K}\sum_{k=1}^{K}[V_{1,\xi}^*(s_1)-V_1^{\pi^k}(s_1)]\leq O(\frac{\eta}{K}\sqrt{H^3SAT\ell})
	\end{equation} 
	We define a new policy $\bar{\pi}$ which uniformly chooses the policy $\pi^k$ for $k\in[K]$. By the occupancy measure method, $V_1^{\pi^k}(s_1)$ is linear in terms of an occupancy measure induced by policy $\pi^k$ and initial state $s_1$, thus:
	\begin{equation}\label{eq_avg_value}
		\frac{1}{K}\sum_{k=1}^{K}V_1^{\pi^k}(s_1)=V_1^{\bar\pi}(s_1)
	\end{equation}
	Combining Eq. \eqref{eq_bound_regret2} with \eqref{eq_avg_value} and recalling our choice for $\eta=\frac{2HI}{\gamma}$, we have
	\begin{equation}\label{eq_bound_regret_final}
		V_1^*(s_1)-V_1^{\bar{\pi}}(s_1)\leq O(I\sqrt{\frac{H^6SA\ell}{K}})
	\end{equation}
\subsection{Bounding the Constraint Violations}
	Similar to Eq. \eqref{eq_bound_regret}, however, instead of Lemma \ref{lem_opt_origin}, \ref{lem_policy} and \ref{lem_sublienar}, using Lemma \ref{lem_opt_approx}, \ref{lem_policy} and \ref{lem_sublienar}, we have
	\begin{equation}\label{eq_bound_violation}
		\frac{1}{K}\sum_{k=1}^{K}[V_{1,\xi}^*(s_1)-V_1^{\pi^k}(s_1)]-\frac{1}{K}\frac{\eta}{I}\sum_{k=1}^{K}\sum_{h=1}^{H}\sum_{i=1}^{I}\mathbf{E}\big[g_{i,\xi}^-(s_h^k,\pi_h^k(s_h^k))\big]\leq O(\frac{\eta}{K}\sqrt{H^3SATl})+\eta H\xi	
	\end{equation} 
	Due to the concavity of the function $x^-=\min\{x,0\}$ and by Jensen Inequality, we have
	\begin{equation}\label{eq_avg_cons}
		\frac{\eta}{I}\frac{1}{K}\sum_{k=1}^{K}\sum_{h=1}^{H}\sum_{i=1}^{I}\mathbf{E}\big[g_{i,\xi}^-(s_h^k,\pi_h^k(s_h^k))\big]\leq \frac{\eta}{I}\sum_{h=1}^{H}\sum_{i=1}^{I}\mathbf{E}\big[g_{i,\xi}^-(s_h,\bar{\pi}_h(s_h))\big]
	\end{equation}
	Combining \eqref{eq_bound_violation} with Eq. \eqref{eq_avg_value} and \eqref{eq_avg_cons}, we have
	\begin{equation}
		V_{1,\xi}^*(s_1)-V_1^{\bar{\pi}}(s_1)-\frac{\eta}{I}\sum_{h=1}^{H}\sum_{i=1}^{I}\mathbf{E}\big[g_{i,\xi}^-(s_h,\bar{\pi}_h(s_h))\big]\leq O(\frac{\eta}{K}\sqrt{H^3SATl})+\eta H\xi
	\end{equation}
	Notice that $\eta=\frac{2HI}{\gamma}$ satisfies the condition that $C^*\geq 2Y^*$ in Lemma \ref{thm.violationgeneral}, thus,
	\begin{equation}\label{eq_bound_violation_final}
		\sum_{h=1}^{H}\sum_{i=1}^{I}\mathbf{E}\bigg|g_{i,\xi}^-(s_h,\bar{\pi}_h(s_h))\bigg|\leq O(\frac{I}{K}\sqrt{H^3SATl})+2HI\xi=O(I\sqrt{\frac{H^4SA\ell}{K})}+2HI\xi
	\end{equation}
\subsection{$\epsilon$-Optimal Policy}
	Choosing $K=\Omega(\frac{I^2H^6SA\ell}{\epsilon^2})$, by Eq. \eqref{eq_bound_regret_final} and \eqref{eq_bound_violation_final}, we have
	\begin{equation}
		\begin{aligned}
		V_1^*(s_1)-V_1^{\bar{\pi}}(s_1)&\leq \epsilon\\
		\sum_{h=1}^{H}\sum_{i=1}^{I}\mathbf{E}\bigg|g_i^-(s_h,\bar{\pi}_h(s_h))\bigg|&\leq \frac{\epsilon}{H}+2HI\xi\\
		\end{aligned}
	\end{equation}
	Recalling the definition of $g_{i,\xi}=f_i^{-}(s,a)+\xi$, we have
	\begin{equation}
		\sum_{h=1}^{H}\sum_{i=1}^{I}\mathbf{E}\bigg|f_i^-(s_h,\bar{\pi}_h(s_h))\bigg|\leq \frac{\epsilon}{H}+3HI\xi
	\end{equation}
	Since $\xi = \frac{\epsilon}{6HI}$, the above shows that we obtain an $\epsilon$-optimal strategy.

\section{Proof of Theorem \ref{cor_deterministic}}\label{app_slow_convergence}
\begin{proof}
	Firstly, let $\xi=0$. Notice that this will break the Slater condition in Lemma \ref{lem_slater}. However, the following proof is not based on the Slater Condition. Moreover, by the choice of $\eta=T^\phi$ and $\phi\in(0,\frac{1}{2})$. The bound for the modified reward function in Lemma \ref{lem_bound_R} holds without any requirement and thus the sub-linear convergence of the \textbf{Modified MDP} in Lemma \ref{lem_sublienar} directly follows. Finally, due to the choice of $\xi=0$, we can simplify the notation $W_{1,\xi}^*\rightarrow W_1^*,W_{1,\xi}^{\pi^k}\rightarrow W_{1}^{\pi^k},g_{i,\xi}^{-}\rightarrow g_{i}^{-}$. With this simplified notation, Lemma \ref{lem_opt_origin} and \ref{lem_policy} hold directly without any change besides notation. Thus, similar to Eq. \eqref{eq_bound_regret}, we have
	\begin{equation}
		\sum_{k=1}^{K}[V_{1}^*(s_1)-V_1^{\pi^k}(s_1)]-\frac{\eta}{I}\sum_{k=1}^{K}\sum_{h=1}^{H}\sum_{i=1}^{I}\mathbf{E}\big[g_i^-(s_h^k,\pi_h^k(s_h^k))\big]\leq \sum_{k=1}^{K}W_{1}^*(s_1)-W_{1}^{\pi^k}(s_1)\leq O(\eta\sqrt{H^3SATl})
	\end{equation}
	Due to the the sum of terms $g_{i}^{-}$ non-positive, we have
	\begin{equation}\label{eq_regret}
		\sum_{k=1}^{K}[V_{1}^*(s_1)-V_1^{\pi^k}(s_1)]\leq O(\eta\sqrt{H^3SATl})=O(T^{\frac{1}{2}+\phi}\sqrt{H^3SA\ell})
	\end{equation}
	Furthermore, notice that $\sum_{k=1}^{K}[V_{1}^*(s_1)-V_1^{\pi^k}(s_1)]\geq -KH=-T$, we have
	\begin{equation}\label{eq_vio}
		\begin{aligned}
		-\frac{\eta}{I}\sum_{k=1}^{K}\sum_{h=1}^{H}\sum_{i=1}^{I}\mathbf{E}\big[g_i^-(s_h^k,\pi_h^k(s_h^k))\big]&\leq O(\eta\sqrt{H^3SATl})+O(T)\\
		\sum_{k=1}^{K}\sum_{h=1}^{H}\sum_{i=1}^{I}\mathbf{E}\big|g_i^-(s_h^k,\pi_h^k(s_h^k))\big|&\leq O(I\sqrt{H^3SATl})+O(\frac{IT}{\eta})=O(IT^{1-\phi})
		\end{aligned}
	\end{equation}
	Due to $\phi\in(0,\frac{1}{2})$, it is obvious that the Eq. \eqref{eq_regret} and \eqref{eq_vio} are both sub-linear, which means that the policy $\pi^k$ in Algorithm \ref{alg:cons_q} converge to the optimal policy $\pi^*$ and the constraint violation of $\pi^k$ converges to 0. 
\end{proof}

\section{Supporting Lemmas from Optimization}\label{app.sec.opt}
	We collect some standard results from the literature for readers' convenience. The following lemmas are mainly from \citep{ding2020provably}, while are extended  to the peak constraint setting rather than average constraint in the prior work. First, we rewrite our constrained problem~\eqref{eq_new_formulation} as,
	\begin{equation}
		\begin{array}{c}
		\max\limits_{\pi}
		\;
		\mathbf{E}\bigg[\sum_{h=1}^{H}r_h(s_h,\pi_h(s_h))\bigg]
		\;\;
		s.t 
		\;\;
		\mathbf{E}\big[g_{i,\xi}(s_h,\pi_h(s_h))\big]\geq 0 \quad\forall h,i
		\end{array}
	\end{equation}
	Let the optimal solution be $\pi^*$ such that
	\begin{equation}
		V_{1,\xi}^{\pi^*}(s_1)
		\;=\;
		\max_{\pi}\;\{ \,V_{1}^{\pi}(s_1) \,\vert\,\mathbf{E}\big[g_{i,\xi}(s_h,\pi_h(s_h))\big]\geq 0 \quad\forall h,i\}.
	\end{equation}
	Let the Lagrangian be
	$\mathcal{L}_\xi(\pi,\bm{\lambda}) := V_{1}^{\pi}(s_1)+ \sum_{h=1}^{H}\sum_{i=1}^{I}\lambda_{h,i}\mathbf{E}\big[g_{i,\xi}(s_h,\pi_h(s_h))\big]$, where $\lambda_{h,i}\geq 0$ is the Lagrange multipliers or dual variables. The associated dual function is defined as 
	\begin{equation}
		\mathcal{D}_{\xi}(\bm{\lambda}):=\max_{\pi}\mathcal{L}_{\xi}(\pi,\bm{\lambda}):= V_{1}^{\pi}(s_1)+ \sum_{h=1}^{H}\sum_{i=1}^{I}\lambda_{h,i}\mathbf{E}\big[g_{i,\xi}(s_h,\pi_h(s_h))\big]
	\end{equation}
	and the optimal dual is $\bm{\lambda}^*:=\arg\min_{\bm{\lambda}\geq 0}\mathcal{D}_{\xi}(\bm{\lambda})$,
	\begin{equation}
		\mathcal{D}_{\xi}(\bm{\lambda}^*):=\min_{\bm{\lambda}\geq 0}\mathcal{D}_{\xi}(\bm{\lambda})
	\end{equation}
	We recall that the problem~\eqref{eq_new_formulation} enjoys strong duality under the standard Slater condition. The proof is a special case of \cite{ding2020provably} in finite-horizon for peak constraint.

\begin{assumption}[Slater Condition]
	There exists $\gamma_{h,i}>0,\forall h,i$ and $\hat{\pi}$ such that $\mathbf{E}\big[g_{i,\xi}(s_h,\hat{\pi}_h(s_h))\big] \geq \gamma_{h,i}$.
\end{assumption}

\begin{lemma}\label{lem_duality}[Strong Duality]\cite{ding2020provably}
	If the Slater condition holds, then the strong duality holds for Eq. \eqref{eq_new_formulation}, 
	\[
	V_{1,\xi}^{\pi^*}(s_1) \;= \; \mathcal{D}_\xi(\bm{\lambda}^*) .
	\]
\end{lemma}
\begin{proof}
	This is an extension of Lemma 7 in \cite{ding2020provably} to the peak constraint case. To prove the strong duality still holds for peak constraint, we construct an equivalent problem to Eq. \eqref{eq_new_formulation} in the average constraint setting. Consider, for each fixed $h\in[H]$ and $i\in[I]$, we define a new function $\tilde{g}_{h',i,\xi}$ in the following way.
	\begin{equation}
		\tilde{g}_{h',i,\xi}(s,a)=\left\{
		\begin{aligned}
		&g_{i,\xi}(s,a) \quad \text{if }h'=h  \\
		&0 \quad \text{otherwise}
		\end{aligned}
		\right.
	\end{equation}
	By this way, the original constraint function can be written as
	\begin{equation}
		\mathbf{E}[g_{i,\xi}(s_h,\pi_h(s_h))]=\mathbf{E}\bigg[\sum_{h'=1}^{H}\tilde{g}_{h',i,\xi}(s_{h'},\pi_{h'}(a_{h'}))\bigg]=:V_{1,\tilde{g}_{h',i,\xi}}^{\pi}(s_1)
	\end{equation}
	Thus, an equivalent problem to Eq. \eqref{eq_new_formulation} is
	\begin{equation}\label{eq:equivalent_problem}
		\begin{array}{c}
		\max\limits_{\pi}
		\;
		\mathbf{E}\bigg[\sum_{h=1}^{H}r_h(s_h,\pi_h(s_h))\bigg]
		\;\;
		s.t 
		\;\;
		V_{1,\tilde{g}_{h',i,\xi}}^{\pi}(s_1)\geq 0 \quad\forall h',i
		\end{array}
	\end{equation}
	It can be seen that Eq. \eqref{eq:equivalent_problem} is the MDP problem with averaged constraints which is the same setting as \cite{ding2020provably}. The strong duality holds for Eq. \eqref{eq:equivalent_problem} and thus holds for the Eq. \eqref{eq_new_formulation}
\end{proof}
It is implied by the strong duality that the optimal solution to the dual problem: $\min_{\bm{\lambda}\geq 0}\mathcal{D}(\bm{\lambda},\xi)$ is obtained at $\lambda^*$. Denote the set of all optimal dual variables as $\Lambda^*$.

Under the Slater condition, an useful property of the dual variable is that the sub-level sets are bounded \cite{ding2020provably}. 

\begin{lemma}[Boundedness of Sublevel Sets of the Dual Function]
	Let the Slater condition hold.
	Fix $C\in\mathbb{R}$. For any $\bm{\lambda} \in\{\bm{\lambda}\geq 0\,\vert\, \mathcal{D}_{\xi}(\bm{\lambda}) \leq C \}$, define $Y:=\max_{h,i}\{\lambda_{h,i}\}$ and $\gamma=\min_{h,i}\gamma_{h,i}$, it holds that 
	\[
	Y\leq\frac{1}{\gamma}\bigg[C -V_{1}^{\hat{\pi}}(s_1)\bigg].
	\]
\end{lemma}
\begin{proof}
	\begin{equation}
	C\geq\mathcal{D}_\xi(\bm{\lambda})\geq V_{1}^{\hat{\pi}}(s_1) + \sum_{h=1}^{H}\sum_{i=1}^{I}\lambda_{h,i}\mathbf{E}\big[g_{i,\xi}(s_h,\hat{\pi}_h(s_h))\big]\geq V_{r}^{\hat{\pi}}(s_1) + Y\gamma
	\end{equation}
	where we utilize the Slater point $\hat{\pi}$ and $Y\leq \sum_{h,i}\lambda_{h,i}$ in the last inequality. We complete the proof by noting $\gamma>0$.
\end{proof}

\begin{corollary}[Boundedness of $Y^\star$]
	\label{cor.boundeddual}
	Let the Slater condition hold, if we take $C = V_{1,\xi}^{\pi^\star}(x_1)$ and by Strong Duality $C=\mathcal{D}_{\xi}(\bm{\lambda}^*)$, then $\Lambda^\star=\{\bm{\lambda}\geq 0\,\vert\, \mathcal{D}_{\xi}(\bm{\lambda})\leq C \}$. Thus, for any $\bm{\lambda}\in\Lambda^\star$, define $Y^*:=\max_{h,i}\{\lambda_{h,i}\}$
	\[
	Y^*\leq\frac{1 }{\gamma}\bigg[V_{1,\xi}^{\pi^*}(s_1)-V_1^{\hat{\pi}} (s_1)\bigg]\leq \frac{H}{\gamma}.
	\]
\end{corollary}
Another useful theorem from the optimization is given as follows. It describes that the constraint violation $\sum_{h=1}^{H}\sum_{i=1}^{I} \mathbf{E}|g_{i,\xi}^-(s_h,\pi_h(s_h))|$ can be bounded if we have some weak bound. We next state and prove it for our problem, which is used in our constraint violation analysis in Appendix \ref{app_thm1}.
\begin{lemma}[Constraint Violation]
	\label{thm.violationgeneral}
	Let the Slater condition hold and $\bm{\lambda}^\star \in \Lambda^\star$. If $C^\star \geq 2 Y^\star$ and for any $\tilde{\pi}$ we have
	\begin{equation}
	V_{1,\xi}^{\pi^*}(s_1)-V_1^{\tilde{\pi}}(s_1)-C^\star\sum_{h=1}^{H}\sum_{i=1}^{I}\mathbf{E}\big[g_{i,\xi}^-(s_h,\tilde{\pi}_h(s_h))\big]\leq\delta
	\end{equation}
	Then,
	\begin{equation}
	\sum_{h=1}^{H}\sum_{i=1}^{I}\mathbf{E}\bigg|g_{i,\xi}^-(s_h,\tilde{\pi}_h(s_h))\bigg|\leq\frac{2\delta}{C^*}
	\end{equation} 
	where $x^{-}=\min(x,0)$. 
\end{lemma}
\begin{proof}
	Define $\bm{\tau}:=\{\tau_{h,i}\}_{h\in[H],i\in[I]}$ and let 
	\begin{equation}
	v_\xi(\bm{\tau}) \triangleq \max_{\pi}\{V_1^{\pi}(s_1)\vert \mathbf{E}\big[g_{i,\xi}(s_h,\pi_h(s_h))\big]\geq\tau_{h,i} \quad\forall h,i\}.
	\end{equation}
	By definition, $v_\xi(\bm{0}) = V_{1,\xi}^{\pi^*}(s_1)$. 
	By the Lagrangian and the strong duality,
	\begin{equation}
	\mathcal{L}_\xi(\pi,\bm{\lambda}^*)\leq\max_{\pi} \mathcal{L}_\xi(\pi,\bm{\lambda}^*)  \;=\; \mathcal{D}_\xi(\bm{\lambda}^*) =V_{1,\xi}^{\pi^*}(s_1)=v_\xi(\bm{0}) 
	\end{equation}
	For any $\tau_{h,i}\in\mathbb{R}$ and $\pi\in\{\pi\vert \mathbf{E}\big[g_{i,\xi}(s_h,\pi_h(s_h))\big]\geq\tau_{h,i} \quad \forall h,i\}$, we have 
	\[
	\begin{array}{rcl}
	v_{\xi}(\bm{0}) - \sum_{h=1}^{H}\sum_{i=1}^{I}\tau_{h,i} \lambda_{h,i}^* &\geq& \mathcal{L}_\xi(\pi,\bm{\lambda}^\star) - \sum_{h=1}^{H}\sum_{i=1}^{I}\tau_{h,i} \lambda_{h,i}^*
	\\[0.2cm]
	&=& V_1^{\pi}(s_1) +\sum_{h=1}^{H}\sum_{i=1}^{I}\lambda_{h,i}^*\mathbf{E}\big[g_{i,\xi}(s_h,\pi_h(s_h))\big] - \sum_{h=1}^{H}\sum_{i=1}^{I}\tau_{h,i} \lambda_{h,i}^*
	\\[0.2cm]
	&=&V_1^{\pi}(s_1) +\sum_{h=1}^{H}\sum_{i=1}^{I}\lambda_{h,i}^*[\mathbf{E}\big[g_{i,\xi}(s_h,\pi_h(s_h))\big] - \tau_{h,i}]
	\\[0.2cm]
	&\geq& V_1^{\pi}(s_1).
	\end{array}
	\]
	If we maximize the right-hand side of above inequality over $\{\pi\vert \mathbf{E}\big[g_{i,\xi}(s_h,\pi_h(s_h))\big]\geq\tau_{h,i} \quad \forall h,i\}$,  we have
	\begin{equation}
	v_\xi(\bm{0}) - \sum_{h=1}^{H}\sum_{i=1}^{I}\tau_{h,i} \lambda_{h,i}^*\geq v_\xi(\bm{\tau})\label{vzeta}
	\end{equation}
	On the other hand, if we define $\bm{\tilde{\tau}}=\{\tilde{\tau}_{h,i}\}$ and $\tilde{\tau}_{h,i}:= \mathbf{E}\big[g_{i,\xi}^-(s_h,\tilde{\pi}_h(s_h))\big]$, notice that the policy $\tilde{\pi}\in\{\pi\vert \mathbf{E}\big[g_{i,\xi}(s_h,\pi_h(s_h))\big]\geq\tilde{\tau}_{h,i}, \forall h,i\}$. Thus,
	\begin{equation}
	V_1^{\tilde{\pi}}(s_1)\leq v_\xi(\tilde{\bm{\tau}}).\label{V1ineq}
	\end{equation}
	Combining \eqref{vzeta} and \eqref{V1ineq}, we have
	\begin{equation}\label{eq_bound_exceed}
	V_1^{\tilde{\pi}}(s_1)-V_{1,\xi}^{\pi^*}(s_1)\leq v_{\xi}(\tilde{\bm{\tau}})-v_{\xi}(\bm{0}) \leq - \sum_{h=1}^{H}\sum_{i=1}^{I}\tilde{\tau}_{h,i} \lambda_{h,i}^*
	\end{equation}
	By the assumption in the statement of the lemma, we have
	\begin{equation}
	V_{1,\xi}^{\pi^*}(s_1)-V_1^{\tilde{\pi}}(s_1)-C^\star\sum_{h=1}^{H}\sum_{i=1}^{I}\mathbf{E}\big[g_{i,\xi}^-(s_h,\tilde{\pi}_h(s_h))\big] \leq\delta
	\end{equation}
	Combined with Eq. \eqref{eq_bound_exceed}, we have
	\begin{equation}
	-C^\star\sum_{h=1}^{H}\sum_{i=1}^{I}\mathbf{E}\big[g_{i,\xi}^-(s_h,\tilde{\pi}_h(s_h))\big]\leq\delta- \sum_{h=1}^{H}\sum_{i=1}^{I}\tilde{\tau}_{h,i} \lambda_{h,i}^*
	\end{equation}
	Recalling the definition of $\tilde{\tau}_{h,i}=\mathbf{E}\big[g_{i,\xi}^-(s_h,\tilde{\pi}_h(s_h))\big]$ above Eq. \eqref{V1ineq}, we have
	\begin{equation}
	\sum_{h=1}^{H}\sum_{i=1}^{I}(\lambda_{h,i}^*-C^*)\mathbf{E}\big[g_{i,\xi}^-(s_h,\tilde{\pi}_h(s_h))\big]\leq\delta
	\end{equation}
	Due to $g_{i,\xi}^-(s_h,\tilde{\pi}_h(s_h))\leq 0$, we have
	\begin{equation}
		\sum_{h=1}^{H}\sum_{i=1}^{I}(C^*-\lambda_{h,i}^*)\mathbf{E}\big| g_{i,\xi}^-(s_h,\tilde{\pi}_h(s_h))\big|\leq\delta
	\end{equation}
	Finally, due to $C\geq 2Y^*=2\max_{h,i}\lambda_{h,i}^*$, we have
	\begin{equation}
	\sum_{h=1}^{H}\sum_{i=1}^{I}\mathbf{E}\bigg| g_{i,\xi}^-(s_h,\tilde{\pi}_h(s_h))\bigg|\leq\frac{2\delta}{C^*}
	\end{equation}
\end{proof}

\bibliography{jmlr_ref}

\end{document}